\pdfoutput=1
\documentclass{shinyart}

\usepackage[utf8]{inputenc}
\usepackage{enumitem} 

\usepackage{shinybib}
\renewcommand{\phi}{\varphi}
\newcommand{\eps}{\varepsilon}
\newcommand{\calA}{\mathcal{A}}
\newcommand{\calI}{\mathcal{I}}
\newcommand{\R}{\mathbb{R}}
\newcommand{\N}{\mathbb{N}}
\newcommand{\1}{\mathbb{1}}
\renewcommand{\bar}{\overline}

\DeclareMathOperator{\proj}{\mathrm{proj}}

\newcommand{\dual}[1]{\langle #1 \rangle}
\newcommand{\norm}[1]{\| #1 \|}
\newcommand{\setof}[2]{\left\{#1:#2\right\}}

\newcommand{\convw}{\rightharpoonup^*}
\newcommand{\conv}{\rightarrow}

\usepackage{enumitem} 
\usepackage{booktabs}
\usepackage[exponent-product=\cdot,retain-zero-exponent=true]{siunitx}
\usepackage{graphicx}
\usepackage{algorithm}
\usepackage[noend]{algpseudocode}

\begin{document}
\title{Quasi-solution of linear inverse problems in non-reflexive Banach spaces}
\author{Christian Clason\thanks{%
        University Duisburg-Essen, Faculty of Mathematics,
        Thea-Leymann-Str.~9,
        45127 Essen, Germany
        (\email{christian.clason@uni-due.de}, \email{andrej.klassen@uni-due.de})
    } 
    \and Andrej Klassen\footnotemark[1]
}
\date{June 26, 2018}
\hypersetup{
    pdftitle = {Quasi-solution of linear inverse problems in non-reflexive Banach spaces},
    pdfauthor = {C.~Clason and A.~Klassen},
}

\maketitle

\begin{abstract}
    We consider the method of quasi-solutions (also referred to as Ivanov regularization) for the regularization of linear ill-posed problems in non-reflexive Banach spaces. Using the equivalence to a metric projection onto the image of the forward operator, it is possible to show regularization properties and to characterize parameter choice rules that lead to a convergent regularization method, which includes the Morozov discrepancy principle. Convergence rates in a suitably chosen Bregman distance can be obtained as well.
    We also address the numerical computation of quasi-solutions to inverse source problems for partial differential equations in $L^\infty(\Omega)$ using a semismooth Newton method and a backtracking line search for the parameter choice according to the discrepancy principle.
    Numerical examples illustrate the behavior of quasi-solutions in this setting.
\end{abstract}

\section{Introduction}

We consider the operator equation $Au=y$ for a linear injective operator $A:U \to Y$ with dense range, where $Y$ is a uniformly convex Banach space and $U$ the -- not necessarily reflexive -- dual of a separable normed vector space. Instead of the exact data $y\in R(A)$ (the range of $A$), usually only a noisy measurement $y^\delta \in Y\setminus R(A)$ with $\norm{y^\delta-y}_Y\leq \delta$ for some noise level $\delta$ is available, and regularization methods have to be applied to obtain a stable (approximate) solution.
Three related strategies that can be found in the literature are:
\begin{itemize}
    \item \emph{Tikhonov regularization} \cite{Tikhonov:1963}, consisting in solving for some $\alpha>0$ the minimization problem
        \begin{equation}\label{eq:tikhonov}
            \tag{T}
            \min_{u\in U} \frac12 \norm{Au-y}_Y^2 + \frac\alpha2 \norm{u}_U^2;
        \end{equation}
    \item \emph{Morozov regularization} (also known as the \emph{method of residuals}) \cite{Morozov}, consisting in solving the constrained minimization problem
        \begin{equation}\label{eq:morozov}
            \tag{M}
            \min_{u\in U} \norm{u}_U \quad \text{subject to}\quad \norm{Au-y}_Y\leq \delta;
        \end{equation}
    \item \emph{Ivanov regularization} (also known as the \emph{method quasi-solutions}) \cite{Iva62}, consisting in solving for some $\rho>0$ the constrained minimization problem
        \begin{equation}\label{eq:ivanov}
            \tag{I}
            \min_{u\in U} \norm{Au-y}_Y \quad\text{subject to} \quad \norm{u}_U\leq \rho.
        \end{equation}
\end{itemize}
It is the last one that is the focus of this work. Let us first briefly comment on the relation of these approaches. Since $A$ is linear, all three minimization problems are convex, and hence their solutions coincide for a suitable choice of $\alpha$ and $\rho$, respectively; see, e.g., \cite[Sec.~3.5]{IvaVasTan02}. However, this is no longer the case for nonlinear operators, as was shown by a counterexample in \cite{LorWor13} (which also, to the best of our knowledge, coined the term \emph{Ivanov regularization}); see also \cite[Ex.~1]{Klassen:2018}. Furthermore, if $U$ is not a Hilbert space, \eqref{eq:tikhonov} is no longer necessarily easier to solve than \eqref{eq:morozov} or \eqref{eq:ivanov}. In fact, for $U=L^\infty(\Omega)$ (which motivates the problem setting stated in the beginning), \cite{CIK:2012} proposed solving \eqref{eq:tikhonov} by transforming it into the form \eqref{eq:ivanov} with an additional penalty $\frac\alpha2\rho^2$ for (the in this case free variable) $\rho$. This motivates directly considering \eqref{eq:ivanov} as a regularization strategy.

A further motivation is parameter identification in partial differential equations, where the domain of the (nonlinear) forward mapping is given by pointwise almost everywhere restrictions on the parameter. If such restrictions already yield a regularization of the ill-posed parameter identification problem, unwanted smoothing from the introduction of a Tikhonov penalty can be avoided. 

In this work, we consider as a first step admissible sets of the form
\begin{equation}\label{eq:Mrho}
    M_\rho := \setof{u \in U}{\norm{u}_{U} \leq \rho}, 
\end{equation}
where we are particularly interested in the case $U= L^\infty(\Omega)$ for a bounded domain $\Omega\subset\R^d$.
We then take as the \emph{quasi-solution} of $Au=y$ the minimizer $u_\rho \in M_\rho \subset U$ of 
\begin{equation}\label{eq:quasisolution} 
    \inf_{u \in M_\rho} \norm{Au - y}_Y,
\end{equation}
where the radius $\rho\geq 0$ plays the role of regularization parameter. Its proper choice is one of the main issues in showing that the method of quasi-solutions can be considered a (convergent) regularization method, and the characterization of such choice rules -- especially of a posteriori-type -- is one of the main contributions of this work. The main difficulty here is the fact that $U$ is a non-reflexive Banach space, which requires working in the weak* topology. This is addressed by considering \eqref{eq:quasisolution} as a projection problem in $Y$, allowing exploiting the strong topology of this space.
We also discuss the numerical solution of \eqref{eq:quasisolution} for the case that $A$ is the solution operator to an elliptic partial differential equation and $U= L^\infty(\Omega)$.

Let us remark on related literature. Quasi-solutions were introduced by Ivanov \cite{Iva62,Iva63,IvaVasTan02} as the stabilization of linear ill-posed problems in normed spaces through a compact restriction of the solution domain based on Tikhonov's theorem \cite{Tikhonov:1963}. 
Further works \cite{DomIva65,Iva69} weakened the requirement to weakly compact norm balls in uniformly convex Banach spaces; see also \cite{VaAg95,IvaVasTan02,Vas08}. More general compact solution domains of monotone or convex and bounded functions were considered in \cite{YaLeTi02,Yagola:2002}.
More recently, convergence rates for quasi-solutions to nonlinear inverse problems with norm balls in Hilbert scales were proved in \cite{NeuRam14}. 
However, the (a posteriori) parameter choice for the radius was not discussed. 
Finally, in parallel to this work, convergence and convergence rates for both Morozov and Ivanov regularization under variational source conditions were derived in a more abstract setting in \cite{Klassen:2018}.
In contrast to this, the current work also treats the numerical computation of quasi-solutions in a non-reflexive Banach space, which to the best of our knowledge has not been done so far. 

This paper is organized as follows. \Cref{sec:projection} collects results on metric projections onto image sets under $A$ that will be needed in the following. In \cref{sec:regularization}, we show that the method of quasi-solutions is indeed a regularization strategy, characterize parameter choice rules leading to a convergent regularization method, and derive convergence rates in a suitably chosen Bregman distance. \Cref{sec:solution} is concerned with the numerical solution of an inverse source problem for a partial differential equation with $U=L^\infty(\Omega)$ based on a semi-smooth Newton method and a line search for the a posteriori choice according to the discrepancy principle. Finally, \cref{sec:example} contains numerical examples illustrating the behavior of quasi-solutions in this setting.

\section{Quasi-solution as metric projection in image space}\label{sec:projection}

To avoid having to work with the weak or weak* topology in $U$, we consider the problem of finding a quasi-solution as a metric projection onto the range of $A$. While existence of a quasi-solution can also be shown by standard methods, the analysis in image space will be helpful for analyzing parameter choice rules.

For this purpose, we make the following assumption on the forward operator.
\begin{assumption}\label{ass:Op_closed}
    The linear operator $A:U\to Y$ 
    \begin{enumerate}[label=(\roman*)]
        \item is injective,
        \item has dense range,
        \item is bounded in the strong topologies of $U$ and $Y$,
        \item is weak*-to-strong closed, i.e., for any sequence $\{u_n\}_{n\in\N}$ with $u_n \convw u$ in $U$ and $Au_n \conv y$ in $Y$ we have that $Au = y$.
    \end{enumerate}
\end{assumption}
From the definition of $M_\rho$ and these assumptions on $A$, we immediately obtain the following useful property.
\begin{lemma}\label{lem:Qrho}
    For every $\rho\geq 0$, the image set $Q_\rho := A(M_\rho)\subset Y$ is non-empty, closed, convex, and bounded.
\end{lemma}
\begin{proof} 
    Non-emptiness, convexity, and boundedness are obvious since $M_\rho$ is a ball and $A$ is linear and bounded. To see the closedness, let $\{y_n\}_{n\in\N}$ be a sequence in $Q_\rho$ converging strongly to some $y\in Y$. By definition of $Q_\rho$, there thus exists for every $n\in \N$ a $u_n\in M_\rho$ such that $Au_n = y_n$.
    Since $\{u_n\}_{n\in \N}\subset M_\rho$ is bounded and $U$ is the dual of a separable vector space, we can apply the Banach--Alaoglu theorem to extract a subsequence, still denoted by $\{u_n\}_{n\in\N}$, with $u_n \convw u\in M_\rho$ (since the unit ball in $U$ is weakly* sequentially closed). By the closedness \cref{ass:Op_closed}\,(iv) on $A$, we thus have $y = Au \in Q_\rho$.
\end{proof}
Since $Y$ was assumed to be a uniformly convex Banach space, we can apply the following result from \cite[Rem.~1.6.2, Thm.~1.6.4]{IvaVasTan02}.
\begin{proposition}\label{prop:projection}
    For every $y\in Y$, there exists a unique metric projection $y_\rho \in Q_\rho$ such that
    \begin{equation}\label{eq:proj_Qrho}
        \norm{y_\rho -y}_Y = \inf_{q\in Q_\rho} \norm{q-y}_Y.  
    \end{equation} 
    Furthermore, the mapping $P_{Q_\rho}:y\mapsto y_\rho$ is continuous.
\end{proposition}
By the definition of $Q_\rho$, this yields the existence of a $u_\rho\in M_\rho$ with $y_\rho=Au_\rho$ that attains the minimum in \eqref{eq:quasisolution}. Together with the injectivity of $A$, we thus immediately obtain the existence of a unique quasi-solution.
\begin{theorem}\label{thm:existence} 
    For every $y\in Y$ and $\rho\geq 0$ there exists a unique solution $u_\rho\in M_\rho$ to \eqref{eq:quasisolution}.
\end{theorem}

We can therefore introduce the \emph{distance function}
\begin{equation}\label{eq:distancefun}
    d:[0,\infty) \times Y \to [0,\infty) \qquad (\rho,y) \mapsto \norm{Au_{\rho} - y}_Y = \min_{u \in M_{\rho}} \norm{Au-y}_Y.     
\end{equation}
The following property of the distance function will be crucial in the following. It was first proved in \cite{DomIva65} for the case that $M_\rho$ is compact, and our extension to weakly* compact and convex subsets follows the line of their proof. (Here we point out that convex and compact sets are in general \emph{not} weakly* compact.)
\begin{proposition}\label{prop:distancefun} 
    For every $y\in Y$, the mapping $\rho \mapsto d(\rho,y)$ is surjective on 
    $(0,\norm{y}_Y)$. 
\end{proposition}
\begin{proof}
    Let $y\in Y$ and $\sigma \in (0,\norm{y}_Y)$ be given. Since the range of $A$ is dense in $Y$, there exists a sufficiently large $\rho$ such that $Q_\rho\cap B_\sigma(y)$ is non-empty. The set of all such $\rho>0$ is thus non-empty and bounded from below and therefore admits an infimum
    \begin{equation*}
        \tilde{\rho} := \inf \bigl\{\rho>0: Q_\rho \cap B_\sigma(y) \neq\varnothing \bigr\}.
    \end{equation*}
    By the definition of the infimum, there thus exists a sequence $\{\rho_n\}_{n\in\mathbb N}\subset (0,\infty)$ with
    \begin{equation*}
        S_n := Q_{\rho_n} \cap B_\sigma(y) \neq \emptyset
    \end{equation*}
    and $\rho_n\to\tilde\rho$. Without loss of generality, we can assume that this sequence is monotonically decreasing, which implies that $Q_{\rho_{n+1}} \subseteq Q_{\rho_{n}}$ and hence that $S_{n+1} \subseteq S_{n}$. 

    We now show that the intersection $S:=\bigcap_{n\in\N} S_n$ is non-empty. First, since every $S_n\subset Y$ is bounded, convex, and closed (as the intersection of such sets) as well as non-empty, we can apply \cref{prop:projection} to obtain for each $n\in \N$ an $s_n := P_{S_n}(0)$. 
    Furthermore, the monotonicity of $\{\rho_n\}_{n\in\N}$ implies that  $\norm{s_n}_Y \leq \rho_1$. Hence, $\{s_n\}_{n\in\N}$ is bounded and, since uniformly convex Banach spaces are reflexive, thus contains a subsequence weakly converging to some $s\in Y$.
    As $S$ is the intersection of convex and closed sets and therefore weakly closed, we obtain that $s\in S$.

    We next show that $s\in Q_{\tilde\rho}$. First, by construction we have that $s\in S_n\subset Q_{\rho_n}$ for every $n\in\N$, i.e., there exists a $u\in M_{\rho_n}$ such that $Au=s$. (Since $A$ is injective, $u$ cannot depend on $n$.) We can thus pass to the limit to obtain that
    \begin{equation*}
        \norm{u}_U \leq \rho_n \to \tilde \rho.
    \end{equation*}
    In fact, we even have that $\norm{u}_U = \tilde\rho$, because if the inequality were strict, we could find a $\hat \rho<\tilde \rho$ such that $s\in Q_{\hat\rho}\cap B_{\sigma}(y)$, in contradiction to the definition of $\tilde \rho$. This shows the claim.

    Similarly, since $s\in S_n$ for all $n\in \N$, we have that $s\in B_\sigma(y)$. Assume now that $\norm{s-y}_Y<\sigma$, i.e., $s$ is an interior point of $B_\sigma(y)$. Then there exists for $h:=-s$ an $\eps>0$ small enough such that $\bar s:= s+\eps h=(1-\eps)s\in B_\sigma(y)$.    
    Since $A$ is linear, we have that $0\in Q_{\tilde\rho}$ and hence by \cref{lem:Qrho} also $\bar s = \eps 0 + (1-\eps)s \in Q_{\tilde\rho}$. 
    (Note that the assumption $0<\sigma<\norm{y}_Y$ implies that $0\notin B_\sigma(y)$ and thus in particular $s\neq 0$.) 
    This implies that there exists a $\bar u \in M_{\tilde \rho}$ such that $A\bar u = \bar s$. 
    But from the linearity and injectivity of $A$ it then follows that $\bar u = (1-\eps)u$ and hence that
    \begin{equation*}
        \norm{\bar u}_U = (1-\eps)\norm{u}_U = (1-\eps)\tilde \rho < \tilde\rho,
    \end{equation*}
    again in contradiction to the definition of $\tilde \rho$. This implies that $\norm{s-y}_Y=\sigma$.

    It remains to show that $\norm{s-y}_Y = \inf_{u\in M_{\tilde\rho}}\norm{Au-y}_Y$. For this, we argue as in the last step: If there exists a $\hat u\in M_{\tilde\rho}$ such that $\norm{A\hat u-y}_Y<\sigma$, we can again find a $\bar u \in M_{\tilde \rho}$ such that $A\bar u\in B_{\sigma}(y)$ but $\norm{\bar u}_U < \tilde \rho$ to contradict the definition of $\tilde \rho$. Hence,
    \begin{equation*}
        d(\tilde \rho,y) = \norm{s-y}_Y = \sigma
    \end{equation*}
    as claimed.
\end{proof} 

The proof of \cref{prop:distancefun} shows in particular that unless $d(\rho,y)=0$ -- i.e., $y\in Q_\rho \subset R(A)$ -- the unique quasi-solution $u_\rho\in M_\rho$ from \cref{thm:existence} always lies on the boundary of $M_\rho$, a result that was already shown in \cite[Prop.~2.2]{NeuRam14} in the case that $U$ and $Y$ are Hilbert spaces.
\begin{corollary}\label{cor:quasisol_boundary}
    For every $\rho\geq 0$ and $y\in Y\setminus Q_\rho$, the solution to \eqref{eq:quasisolution} satisfies $\norm{u_\rho}_U = \rho$.
\end{corollary}
This result also implies a strict monotonicity of the distance function.
\begin{lemma}\label{lem:d_monotone}
    For every $y\in Y\setminus R(A)$, the mapping $\rho \mapsto d(\rho,y)$ is strictly monotonically decreasing and continuous.
\end{lemma}
\begin{proof}
    If $\rho_1 \leq \rho_2$, the definition of the distance function and the fact that $M_{\rho_1}\subset M_{\rho_2}$ immediately imply that
    \begin{equation*}
        d(\rho_2,y)=\inf_{u \in M_{\rho_2}} \norm{Au-y}_Y \leq \inf_{u \in M_{\rho_1}} \norm{Au-y}_Y = d(\rho_1,y).
    \end{equation*}
    Assume now that $\rho_1< \rho_2$ but $d(\rho_1,y)=d(\rho_2,y)$. By \cref{cor:quasisol_boundary}, we have that $\norm{u_{\rho_1}}_U=\rho_1 < \rho_2 = \norm{u_{\rho_2}}_U$ and in particular $u_{\rho_1}\neq u_{\rho_2}$. 
    But by injectivity of $A$, we have that $Au_{\rho_1} \neq Au_{\rho_2}$ as well, and thus the uniform convexity of $Y$ implies that $\bar u := \frac12(u_{\rho_1}+u_{\rho_2})\in M_{\rho_2}$ satisfies
    \begin{equation*}
        \norm{A\bar u -y} = \norm{\tfrac12(Au_{\rho_1} + Au_{\rho_2})-y}_Y < \frac12 \left(\norm{Au_{\rho_1} -y}_Y +\norm{Au_{\rho_2} -y}_Y \right) = d(\rho_2,y),
    \end{equation*}
    in contradiction to the definition of the distance function.

    Together with the surjectivity from \cref{prop:distancefun}, the strict monotonicity implies the continuity.
\end{proof}

Finally, we will also require the continuity of the distance function with respect to its second argument.
\begin{lemma}\label{lem:d_lipschitz}
    For each $\rho\geq 0$, the mapping $y\mapsto d(\rho,y)$ is non-expansive (i.e., Lipschitz continuous with constant $1$).
\end{lemma}
\begin{proof} 
    Consider $y_1,y_2 \in Y$ and corresponding projections $q_1 := P_{Q_\rho}y_1$ and $q_2 := P_{Q_\rho}y_2$. The optimality property of the projections then imply that
    \begin{align*}
        d(\rho,y_1) &= \norm{q_1 - y_1}_Y \leq \norm{q_2 - y_1}_Y \leq \norm{q_2 - y_2}_Y + \norm{y_2 - y_1}_Y, \\
        d(\rho,y_2) &= \norm{q_2 - y_2}_Y \leq \norm{q_1-y_2}_Y \leq \norm{q_1 - y_1}_Y + \norm{y_1 - y_2}_Y,
    \end{align*}
    which together yield that 
    \begin{equation*}
        |d(\rho,y_1) - d(\rho,y_2)| \leq \norm{y_1 - y_2}_Y.
        \qedhere
    \end{equation*}
\end{proof}

\section{Regularization and parameter choice}\label{sec:regularization}

In this section, we show that the method of quasi-solutions is a regularization in the classical sense \cite{EngHanNeu96,SchKalHofKaz10} and leads -- combined with an appropriate parameter choice rule for $\rho$ -- to a convergent regularization method. Since in this case the solution domain $M_\rho$ is not compact and $U$ is not assumed to be reflexive, we can only expect weak* convergence.

\subsection{Stability and convergence}

To show that quasi-solutions are a regularization of the ill-posed operator equation $Au=y$, we first need to show that for each $\rho\geq 0$, the solution to \eqref{eq:quasisolution} depends continuously on $y\in Y$. We again exploit the equivalence of \eqref{eq:quasisolution} with the metric projection onto $Q_\rho:= A(M_\rho)$.
\begin{theorem}\label{thm:stability}
    Let $\rho\geq 0$ and $y\in Y$ be given. Then for every sequence $\{y_n\}_{n\in\N}\subset Y$ with $y_n\to y$, the corresponding sequence of solutions $\{u_n\}_{n\in\N}$ to \eqref{eq:quasisolution} for $y_n$ in place of $y$ converges weakly* to $u_\rho$.
\end{theorem}
\begin{proof} 
    First, \cref{prop:projection} yields for every $n\in\N$ a unique metric projection $q_n\in Q_\rho$ and hence, by definition of $Q_\rho$, a $u_n\in M_\rho$ such that $Au_n = q_n$. Furthermore, by continuity of the metric projection, we have that $q_n\to q:= P_{Q_\rho} y$. As in the proof of \cref{lem:Qrho}, we can now extract a subsequence, still denoted by $\{u_n\}_{n\in\N}$, such that $u_n \convw u\in M_\rho$ with $Au=q$, i.e., $u\in M_\rho$ minimizes $\norm{Au-y}_Y$ over $M_\rho$. This shows that $u=u_\rho$ is the solution to \eqref{eq:quasisolution}. Since this solution is unique by \cref{thm:existence}, a subsequence--subsequence argument yields weak* convergence of the full sequence.
\end{proof}

In the case of Ivanov regularization, convergence as $\rho\to\infty$ for fixed $y\in R(A)$ is obvious since for every $\rho\geq \norm{u^\dag}_U$ with $Au^\dag = y$, the choice $u_\rho = u^\dag\in M_\rho$ clearly solves \eqref{eq:quasisolution}. Nevertheless, we state the consequence as a theorem for the sake of completeness.
\begin{theorem}\label{thm:approximation}
    Let $y\in R(A)$ be given and $u^\dag\in U$ such that $Au^\dag = y$. Then for every sequence $\{\rho_n\}_{n\in\N}$ with $\rho_n\to \infty$, we have $u_{\rho_n}\to u^\dag$.
\end{theorem}

\subsection{Parameter choice}

We now address the convergence of Ivanov regularization in combination with a parameter choice strategies for $\rho$. As we will see, there are fundamental differences to Tikhonov (and Morozov) regularization, and hence the results of this section are one of the main contributions for this work.

We begin by characterizing parameter choices that lead to a convergent regularization method. In the following, let $y\in R(A)$ be arbitrary, $\{\delta_n\}_{n\in\N}$ be a non-negative sequence with $\delta_n \to 0$, and $y^{\delta_n} \in Y$ with $\norm{y^{\delta_n}-y}_Y\leq \delta_n$ for each $n\in \N$. Let further $u^\dag\in U$ denote the (unique) solution to $Au=y$ and for given $\rho_n\geq 0$, let $u_{\rho_n}^{\delta_n}\in U$ denote the solution to 
\begin{equation}
    \min_{u\in M_{\rho_n}} \norm{Au-y^{\delta_n}}_Y.
\end{equation}
\begin{proposition}\label{prop:necessary_sufficient}
    If $A$ is not continuously invertible, the sequence $\{u_{\rho_n}^{\delta_n}\}_{n\in\N}$ weakly* converges to $u^\dag$ as $n\to\infty$ if and only if the following two conditions hold:
    \begin{enumerate}[label=(\roman*)]
        \item $\{\rho_n\}_{n\in\N}$ is bounded,
        \item $\liminf\limits_{n\to\infty} \rho_n \geq \norm{u^\dag}_U$.
    \end{enumerate}
\end{proposition}
\begin{proof} 
    We first show that these conditions are sufficient by proceeding similarly to the proof of \cref{thm:stability}. 
    First, condition (i) implies that $\{u_{\rho_n}^{\delta_n}\}_{n\in\N}$ is bounded, and hence we can extract a subsequence $u_n\convw u$ for some $u\in U$. By passing to a further subsequence, we can in addition assume that $\rho_n \to \hat\rho \geq \norm{u^\dag}_U$ by condition (ii). This implies that
    \begin{equation}
        \norm{P_{Q_{\rho_n}}y-y}_Y = d(\rho_n,y) \to d(\hat\rho,y) = 0,
    \end{equation}
    where we have used the continuity of the distance function from \cref{lem:d_lipschitz} and the fact that $Au^\dag = y\in Q_{\hat \rho}$. Furthermore, for every such $n$, \cref{prop:projection} yields the existence of a unique metric projection $P_{Q_{\rho_n}} y^{\delta_n} = Au_{\rho_n}^{\delta_n}$ which by definition satisfies
    \begin{equation*}
        \norm{P_{Q_{\rho_n}}y^{\delta_n} - y^{\delta_n}}_Y = \inf_{q \in Q_{\rho_n}} \norm{q - y^{\delta_n}}_Y \leq \norm{P_{Q_{\rho_n}}y - y^{\delta_n}}_Y.
    \end{equation*}
    Combining the above, we obtain that
    \begin{equation*}
        \begin{aligned}
            \norm{P_{Q_{\rho_n}}y^{\delta_n} - y}_Y &\leq \norm{P_{Q_{\rho_n}}y^{\delta_n} - y^{\delta_n}}_Y + \norm{y^{\delta_n} - y}_Y\\
                                                    &\leq \norm{P_{Q_{\rho_n}}y - y^{\delta_n}}_Y + \delta_n \\
                                                    &\leq \norm{P_{Q_{\rho_n}}y - y}_Y + 2\delta_n.
        \end{aligned}
    \end{equation*}
    Passing to the limit then yields that $P_{Q_{\rho_n}}y^{\delta_n} \to y$ and hence by the closedness that $Au=y$, i.e., $u=u^\dag$. Convergence of the full sequence now follows again from a subsequence--subsequence argument.

    To show the necessity, first note that by assumption $R(A)$ is not closed and hence has no interior points (otherwise the algebraic interior of $R(A)$ would not be empty either, and the linearity of $A$ would lead to the contradiction $R(A)=Y$.) For any $n\in\N$, we can thus find a $y^{\delta_n}\in B_{\delta_n}(y) \setminus R(A)$. \Cref{cor:quasisol_boundary} then implies that the corresponding quasi-solution satisfies $\norm{u_{\rho_n}^{\delta_n}}_U = \rho_n$. Hence, if $\{\rho_n\}_{n\in\N}$ is unbounded, $\{u_{\rho_n}^{\delta_n}\}_{n\in\N}$ is unbounded as well and therefore cannot weakly* converge (since weakly* convergent sequences are bounded by the Banach--Steinhaus Theorem). On the other hand, if $u_{\rho_n}^{\delta_n}\convw u^\dag$, then the weak* lower semi-continuity of the norm immediately implies that
    \begin{equation*}
        \norm{u^\dag}_U \leq \liminf_{n\to\infty} \norm{u_{\rho_n}^{\delta_n}}_U = \liminf_{n\to\infty} \rho_n.
        \qedhere
    \end{equation*}
\end{proof} 

\Cref{prop:necessary_sufficient} shows in particular that Ivanov regularization converges under the constant parameter choice $\rho_n \equiv \norm{u^\dag}_U$.
However, this is not a valid parameter choice in the strict sense, which may only depend on $\delta$ and $y^\delta$, but not on the exact solution $u^\dag$. In fact, \cref{prop:necessary_sufficient} implies that there can be no convergent a priori parameter choice $\rho = \rho(\delta)$ unless the inverse problem is well-posed.
\begin{corollary}\label{thm:a-priori}
    Let $\rho(\delta)$ be an a priori parameter choice. If $u_{\rho_n}^{\delta_n}\convw u^\dag$ for $\rho_n = \rho(\delta_n)$ with $\delta_n\to 0$ as $n\to\infty$, then $A$ is continuously invertible.
\end{corollary}
\begin{proof}
    Since the claim must hold independently of $u^\dag$, condition (ii) of \cref{prop:necessary_sufficient} has to be satisfied for any $u^\dag\in U$. This implies a fortiori that $\{\rho_n\}_{n\in\N}$ is unbounded, in contradiction to condition (i). Hence, $A$ must be continuously invertible.
\end{proof}
Of course, under the additional a priori information $u^\dag \in M_{\rho^\dag}$ for some $\rho^\dag\geq 0$, the choice $\rho_n\equiv \rho^\dag$ leads to a convergent regularization method (cf.~\cite{NeuRam14}).

This leaves a posteriori choice rules $\rho=\rho(\delta,y^\delta)$. Here we consider the Morozov discrepancy principle \cite{Mor84}, which for some $\tau >1$ chooses $\rho(\delta,y^\delta)$ such that
\begin{equation}\label{eq:discrepancy}
    \delta \leq \norm{Au_{\rho(\delta,y^\delta)}^\delta - y^\delta}_Y \leq \tau\delta.
\end{equation}
This choice indeed leads to a convergent method.
\begin{theorem}\label{thm:morozov}    
    Let $\rho_n:=\rho(\delta_n,y^{\delta_n})$ be chosen according to \eqref{eq:discrepancy} for some $\tau>1$. Then $u_{\rho_n}^{\delta_n}\convw u^\dag$ as $n\to\infty$.
\end{theorem}
\begin{proof}
    We make use of the properties of the distance function \eqref{eq:distancefun} from \cref{sec:projection}. First note that for any $0<\delta_n<\norm{y^{\delta_n}}_Y$ (which holds for $n$ sufficiently large provided $\norm{y}_Y>0$), \cref{prop:distancefun} yields the existence of a $\rho_n$ satisfying \eqref{eq:discrepancy}. 
    It remains to verify the conditions of \cref{prop:necessary_sufficient}.

    We first show that $\{\rho_n\}_{n\in\N}$ is bounded by $\norm{u^\dag}_U$. Assume that there exists an $n\in \N$ such that $\rho_n > \norm{u^\dag}_U$, which implies that $u^\dag \in M_{\rho_n}$.
    Since $\norm{u_{\rho_n}^{\delta_n}}_U = \rho_n$ for $\delta_n>0$ by \cref{cor:quasisol_boundary}, we further have that $u_{\rho_n}^{\delta_n}\neq u^\dag$. Hence it follows from the definition  and uniqueness of the quasi-solution that
    \begin{equation*}
        \norm{Au_{\rho(\delta,y^\delta)}^\delta - y^\delta}_Y = \min_{u\in M_{\rho_n}}\norm{Au-y^{\delta_n}}_Y < \norm{Au^\dag -y^{\delta_n}}_Y \leq \delta_n,
    \end{equation*}
    in contradiction to the choice of $\rho_n$.

    Assume now that
    \begin{equation*}
        \liminf_{n \rightarrow \infty} \rho_n < \norm{u^\dag}_U=:\rho^\dag.
    \end{equation*}
    Since the inequality is strict, there exists -- after passing to a subsequence -- for any $\eps>0$ an $N \in \N$ such that $\rho_n < \tilde\rho:=\norm{u^\dag}_U - \eps$ for all $n > N$. From \cref{lem:d_monotone} we thus obtain that $d(\rho_n,y^{\delta_n}) > d(\tilde{\rho},y^{\delta_n})$ for all $n>N$.
    The parameter choice rule \eqref{eq:discrepancy} and \cref{lem:d_lipschitz} then imply that
    \begin{equation*}
        \lim_{n \to \infty} \tau\delta_n \geq \lim_{n \to \infty} d(\rho_n,y^{\delta_n}) \geq \lim_{n \to \infty} d(\tilde{\rho},y^{\delta_n}) = d(\tilde{\rho},y) > 0
    \end{equation*}
    since $u^\dag\notin M_{\tilde\rho}$, in contradiction to $\delta_n\to 0$.
\end{proof}
The proof of \cref{thm:morozov} in fact shows that for $\rho_n:=\rho(\delta_n,y^{\delta_n})$ chosen according to the discrepancy, we have $\rho_n\to \norm{u^\dag}_U$ from below.

\subsection{Convergence rates}

Since we have only shown weak* convergence of $u_\rho^\delta$ to $u^\dag$ as $\delta\to 0$, we cannot expect convergence rates for the strong error $\norm{u_\rho^\delta-u^\dag}_U$. As usual for inverse problems in Banach spaces, we thus consider rates for the error measured by the Bregman distance; see, e.g., \cite{BurOsh04,Res05,HoKaPoeSche07,Flemming:2010,Grasmair:2010}. We recall that for a convex functional $J:U\to \R\cup\{\infty\}$, the convex subdifferential of $J$ at $u\in U$ with $J(u)<\infty$ is given by
\begin{equation}
    \partial J(u) = \setof{\xi\in U^*}{\dual{\xi,v-u}_U \leq J(v)-J(u)\quad\text{for all }v\in U}.
\end{equation}
Note that the subdifferential can be empty unless $J(u)<\infty$ for all $u\in U$. For given $u\in U$ and $\xi \in \partial J(u)$, we can now define the Bregman distance
\begin{equation}
    D_J^\xi(v;u) = J(v)-J(u) - \dual{\xi,v-u}_U\qquad\text{for all }v\in U.
\end{equation}
By the definitions, we have $D_J^\xi(v;u) \geq0$ for any $v\in U$ as well $D_J^\xi(u; u)=0$, which justifies using the Bregman distance to measure the error. Here, we choose $J(u) := \norm{u}_U$ and point out that this choice is different from the classical one as suggested in \cite{BurOsh04}; to fit into their framework, we would have to take $J(u)=\delta_{M_\rho}(u)$ with $\delta_{M_\rho}(u) = 0$ if $u\in M_\rho$ and $\infty$ else, which however would make the Bregman distance uninformative for $v\notin M_\rho$.

Together with a classical source condition, we can then derive the expected convergence rate under the a posteriori choice rule \eqref{eq:discrepancy}.
\begin{proposition}\label{thm:conv_rates} 
    Assume there exists a $w \in Y^*$ with $\xi := A^*w\in \partial J(u^\dag)$. If $\rho=\rho(\delta,y^\delta)$ is chosen according to \eqref{eq:discrepancy}, then 
    \begin{equation}\label{eq:conv_rates}
        D_J^\xi(u_\rho^\delta;u^\dagger) \leq C\delta.
    \end{equation} 
\end{proposition}
\begin{proof}
    First note that the proof of \cref{thm:morozov} together with \cref{cor:quasisol_boundary} implies that
    \begin{equation*}
        \norm{u_\rho}_U=\rho \leq \norm{u^\dag}_U.
    \end{equation*}
    Hence, we can use the source condition together with the parameter choice rule to estimate
    \begin{equation*}
        \begin{aligned}[b]
            D^\xi_J(u_\rho;u^\dag) 
            & = \norm{u_\rho}_{U} - \norm{u^\dag}_{U}- \dual{A^*w, u_\rho - u^\dag}_{U}  \\
            & \leq \left|\dual{ A^*w , u_\rho - u^\dagger }_{U}\right| = \left|\dual{ w , A(u_\rho - u^\dagger)}_{Y}\right| \\
            & \leq \norm{w}_{Y^*} \norm{Au_\rho - Au^\dagger}_{Y}  \\
            & \leq \norm{w}_{Y^*} \left(\norm{Au_\rho - y^\delta}_Y + \norm{y^\delta - y}_Y \right)\\
            & \leq \norm{w}_{Y^*}(\tau + 1)\delta.
        \end{aligned}
        \qedhere
    \end{equation*} 
\end{proof}

\begin{remark}\label{rem:bregman}
    We end this section by remarking on the interpretation of the convergence in Bregman distance in our context. Recall that the subgradient $\xi$ in \eqref{eq:conv_rates} satisfies by definition.
    \begin{equation}\label{eq:subgrad}
        \xi \in \partial \norm{\cdot}_{L^\infty}(u^\dag) = \setof{\xi\in L^\infty(\Omega)^*}{\dual{ \xi , u^\dag }_{L^\infty(\Omega)} = \norm{u^\dag}_{L^\infty(\Omega)},
        \norm{\xi}_{L^\infty(\Omega)^*} = 1}.
    \end{equation}
    Assume now that the set $C:=\setof{x\in\Omega}{|u(x)|=\rho}$ is measurable. 
    In this case, it is straightforward to verify that one possible choice is
    $\xi \in  L^1(\Omega) \subset (L^\infty)^*$
    defined pointwise almost everywhere by
    \begin{equation}\label{eq:xi}
        \xi(x) :=
        \begin{cases}
            \phantom{-}|C|^{-1} & \text{if }  \setof{x \in \Omega}{u^\dag(x) = \rho} \\
            \phantom{-}0 & \text{if }  \setof{x \in \Omega}{|u^\dag(x)| < \rho} \\
            -|C|^{-1} & \text{if }  \setof{x \in \Omega}{u^\dag(x) = -\rho} 
        \end{cases},
    \end{equation}
    where $|C|$ denotes the Lebesgue measure of $C$.
    In this case, convergence in Bregman distance entails pointwise convergence $u^\delta_\rho(x)\to u^\dag(x)$ on the active set where $|u^\dag(x)| = \rho$. Of course, this choice does not satisfy the source condition $\xi\in R(A^*)\subset H^1(\Omega)$, so \cref{thm:conv_rates} does not imply a pointwise convergence on this set with rate $\mathcal{O}(\delta)$.
\end{remark}

\section{Numerical solution}\label{sec:solution}

We now address the numerical computation of the quasi-solution problem \eqref{eq:quasisolution} for the case $U=L^\infty(\Omega)$ and $Y=L^2(\Omega)$ for a bounded Lipschitz domain $\Omega\subset \R^d$, $d\in\{1,2,3\}$, using a combination of a semi-smooth Newton method and a backtracking procedure for the a posteriori parameter choice of $\rho$. As a model problem, we consider the identification of the source term $u\in L^\infty(\Omega)$ from observation of the state $y\in H^1(\Omega)$ given as the weak solution to
\begin{equation}\label{pde}
    \left\{
        \begin{aligned}
            -\Delta y + cy & = u \quad\text{in }\Omega, \\
            \partial_\nu y &= 0 \quad\text{on }\partial\Omega,
        \end{aligned}
    \right.
\end{equation}
for given $c> 0$. Setting now
\begin{equation*}
    A:L^\infty(\Omega)\to L^2(\Omega),\qquad u\mapsto y,
\end{equation*}
it is easy to see that $A$ is linear, injective, continuous, and weakly$*$ closed. It is well-known from PDE-constrained optimization that in this case $u_\rho$ is a solution to \eqref{eq:quasisolution} (which is known in that context as a \emph{bang-bang control problem}, see, e.g., \cite{Hinze:2012}) if and only if
\begin{equation}\label{eq:solution:projection}
    u_\rho = \proj_{M_\rho}(u_\rho - p_\rho),
\end{equation}
where 
\begin{equation*}
    \proj_{M_\rho}:L^2(\Omega)\to L^2(\Omega),\qquad
    [\proj_{M_\rho}(v)](x) =: 
    \begin{cases}
        \rho  & \text{if} \quad  v(x)    >  \rho \\
        v(x) & \text{if} \quad |v(x)| \leq \rho \\
        -\rho & \text{if} \quad  v(x)    < -\rho 
    \end{cases} 
\end{equation*}
and $p_\rho$ is the solution to the \emph{adjoint problem}
\begin{equation}\label{pde_adjoint}
    \left\{
        \begin{aligned}
            -\Delta p + cp &= f \quad\text{in }\Omega, \\
            \partial_\nu p &= 0 \quad\text{on }\partial\Omega,
        \end{aligned}
    \right.
\end{equation}
for $f=Au_\rho-y^\delta$; see, e.g., \cite[Lem.~2.26]{TroBook}. (Here we rely on the fact that the adjoint state $p_\rho\in H^1(\Omega)$ to allow a pointwise almost everywhere evaluation.)

We now wish to solve \eqref{eq:solution:projection} using a locally superlinearly convergent semi-smooth Newton method \cite{Hintermuller:2002a,Ulbrich:2011}. However, it is known that the pointwise projection is semi-smooth from $L^p(\Omega)\to L^q(\Omega)$ if and only if $p>q$, and hence \eqref{eq:solution:projection} is not semi-smooth with respect to $u_\rho$. Usually, this is addressed by including additional (small) Tikhonov regularization in \eqref{eq:quasisolution}, which allows canceling $u_\rho$ inside the projection; see, e.g., \cite{Hintermuller:2002a,Hinze2005} as well as \cite[Thm.~2.28, Ch.~2.12.4]{TroBook}. Since the focus of this work is Ivanov regularization, we do not follow this route here and instead consider a finite element discretization of \eqref{eq:quasisolution}.

Let $Y_h \subset L^2(\Omega)$ be a finite-dimensional space spanned by the usual continuous piecewise linear nodal basis (``hat'') functions based on the vertices $\{x_j\}_{j=1}^{N}$ of a triangulation of $\overline\Omega$. Taking now $y_h,u_h,p_h\in Y_h$ (which we identify with the vector of their basis coefficients, which coincide with the values at the nodes in this case), noting that functions in $Y_h$ attain their maximum and minimum at the nodes, and introducing the stiffness matrix $K_h\in \R^{N\times N}$ and the mass matrix $M_h\in \R^{N\times N}$, we obtain the discretized conditions
\begin{equation}
    \label{eq:discretesystem}
    \left\{
        \begin{aligned}
            (K_h+cM_h) y_h &= M_h u_h,\\
            (K_h+cM_h) p_h &= M_h (y_h-y^\delta_h),\\
            u_h &= \proj_{[-\rho,\rho]}(u_h-p_h),
        \end{aligned}
    \right.
\end{equation}
where $y^\delta_h\in Y_h$ denotes the $L^2$ projection of $y^\delta$ and the projection is now understood as componentwise. This is now semi-smooth from $\R^N$ to $\R^N$, where a Newton derivative at $v\in \R^N$ direction $w\in \R^N$ is given componentwise by
\begin{equation}
    [D_N \proj_{[-\rho,\rho]}(v)w]_i =
    \begin{cases}
        w_i  \quad & \text{if }  |w_i| \leq \rho, \\
        0  \quad & \text{else.} 
    \end{cases}
\end{equation}  
Introducing the active and inactive sets as 
\begin{equation}\label{eq:activeset}
    \left\{
        \begin{aligned}
            \calA_+ & :=  \setof{i}{u_i - p_i > \rho}    \\
            \calA_- & :=  \setof{i}{u_i - p_i <-\rho}    \\
            \calI & :=  \setof{i}{|u_i - p_i| \leq \rho}   
        \end{aligned}
    \right.
\end{equation}
and the corresponding characteristic functions via
$[\1_{\calA}]_i = 1$ if $i\in \calA$ and $0$ else, we can write a Newton step as 
\begin{multline}\label{newton_step}
    B_k := 
    \begin{pmatrix}
        K_h+cM_h  &    O        & -M_h   \\
        -M_h   &  K_h+cM_h      &  O    \\
        O   &  \1_{\calA_+^k\cup\calA_-^k}  & \1_{\calI^k}-\1_{\calA_+^k\cup\calA_-^k}
    \end{pmatrix}
    \begin{pmatrix}
        {\delta y}  \\
        {\delta p}   \\
        {\delta u}  
    \end{pmatrix}
    \\
    =  -
    \begin{pmatrix}
        (K_h+cM_h)y_h^k - M_hu_h^k        \\
        (K_h+cM_h)p_h^k - M_h(y_h^k-y^\delta_h)    \\
        u_h-\proj_{[-\rho,\rho]}(u_h^k-p_h^k)
    \end{pmatrix}
    =:  - {b^k} 
\end{multline}
and setting $y_h^{k+1}=y_h^k+\delta y$, $p_h^{k+1}=p_h^k+\delta p$, and $u_h^{k+1}=u_h^k+\delta u$. The iteration is terminated if either a fixed number $k_{\max}$ of iterations is exceeded or if both $\calA_+^k=\calA_+^{k-1}$ and $\calA_-^k=\calA_-^{k-1}$. In the latter case, $(p_h^k,y_h^k,u_h^k)$ is already a solution to \eqref{eq:discretesystem}; see \cite[Rem.~7.1.1]{Kunisch:2008a}.
To improve the robustness of the procedure, we also include dampening, which requires an additional criterion that $\norm{b^k}_2\leq \mathrm{tol}$ for some small $\mathrm{tol}>0$ for successful termination. The full procedure is given in Algorithm \ref{alg:ssn}.
\begin{algorithm}[ht]
    \caption{Damped semi-smooth Newton method}\label{alg:ssn}
    \begin{algorithmic}[1]
        \Require {$y^\delta_h$, $\rho>0$, $q\in(0,1)$, $i_{\max}$, $k_{\max}$, $\mathrm{tol}>0$, $y^0_h$, $p_h^0$, $u_h^0$}
        \For{$k=1,\dots,k_{\max}$}
        \State compute active sets $\calA_+^k,\calA_-^k,\calI^k$ from \eqref{eq:activeset}
        \State compute Newton matrix $B_k$, right-hand side $b^k$ from \eqref{newton_step}
        \If{$\norm{b_k}_2 < \mathrm{tol}$}
        \State set $\mathrm{conv}=\mathrm{true}$;
        \Return
        \EndIf
        \State solve Newton step \eqref{newton_step} for $(\delta y,\delta p,\delta u)$
        \For {$i=1,\dots,i_{\max}$}
        \State set $y_h^i + q^{i-1} \delta y$, $p_h^i + q^{i-1} \delta p$, $u_h^i + q^{i-1} \delta u$
        \State compute $b^i$ from \eqref{newton_step}
        \If{$\norm{b^i}_h < \norm{b^k}_h$}
        \State set $y_h^{k+1} = y_h^i$, $p_h^{k+1} = p_h^i$, $u_h^{k+1} = u_h^i$; \textbf{break}
        \EndIf
        \EndFor
        \EndFor
        \State set $\mathrm{conv}=\mathrm{false}$
        \Ensure{$y_h^k,p_h^k,u_h^k$, $\mathrm{conv}$}
    \end{algorithmic}
\end{algorithm}

In practice, we have to deal with the fact that Newton methods converge only locally. We therefore use a backtracking-type method for computing a regularization parameter $\rho$ satisfying the discrepancy principle \eqref{eq:discrepancy} as a continuation strategy. Our approach is based on the following two observations: First, since we are dealing with a discretized problem, the operator $A_h:=(K_h+cM_h)^{-1}$ is invertible and thus any $y^\delta_h$ is in the range of $A_h$. There thus exists a $\hat \rho$ such that the constraint in the discrete version of \eqref{eq:quasisolution} is inactive and the semi-smooth Newton method will therefore converge in a single step. Second, the solution for a given $\rho_k$ will be a good starting point for $\rho_{k+1}<\rho_k$ provided the difference is not too large. We thus proceed in three phases:
\begin{enumerate}
    \item[I.] Starting from some $\rho_0\geq 0$, we quickly increase $\rho_k$ until the Newton method converged in a fixed number of iterations and the residual is smaller than the noise level.
    \item[II.] We then quickly decrease $\rho_k$ until the Newton method fails to converge.
    \item[III.] Starting from the last converged iteration, we adaptively decrease $\rho_k$ further to ensure convergence of the Newton method, until the discrepancy principle \eqref{eq:discrepancy} is satisfied.
\end{enumerate}
The full procedure is given in Algorithm \ref{alg:param}. Of course, we also terminate the iteration if the discrepancy principle is already satisfied in phase I or II. Furthermore, to reduce the number of computed quasi-solutions, the increase of $\rho_k$ in phase I can be made adaptive as well.
\begin{algorithm}[ht]
    \caption{Parameter choice}\label{alg:param}
    \begin{algorithmic}[1]
        \Require{$y^\delta$, $\delta$, $\tau>1$, $\rho_0\geq 0$}
        \State set $y^k_h=0$, $p^k_h=0$, $u^k_h=0$
        \For{$k=0,\dots$}\Comment{Phase I}
        \State compute $y^k_h,p^k_h,u^k_h,\mathrm{conv}$ using \cref{alg:ssn} for $\rho_k$, $y^{k-1}_h,p^{k-1}_h,u^{k-1}_h$
        \State compute $d_k = \norm{y_k - y^\delta}_2$
        \If{$d_k<\delta$ and $\mathrm{conv}=\mathrm{true}$}
        \State \textbf{break}
        \EndIf
        \State set $\rho_{k+1} = \rho_k + \rho_0$
        \EndFor
        \State set $\rho_0 = \rho_k/2$, $y^0_h = y^k_h$, $p^0_h = p^k_h$, $u^0_h = u^k_h$
        \For{$k=0,\dots$}\Comment{Phase II}
        \State compute $y^k_h,p^k_h,u^k_h,\mathrm{conv}$ using \cref{alg:ssn} for $\rho_k$, $y^{k-1}_h,p^{k-1}_h,u^{k-1}_h$
        \State compute $d_k = \norm{y_k - y^\delta}_2$
        \If{$d_k>\delta$ or $\mathrm{conv}=\mathrm{false}$}
        \State \textbf{break}
        \EndIf
        \State set $\rho_{k+1} = \rho_k/2$
        \EndFor
        \State set $\Delta\rho=\rho_{k-1}/2$, $\rho_0 = \rho_{k-1}-\Delta\rho$, $y^0_h = y^k_h$, $p^0_h = p^k_h$, $u^0_h = u^k_h$
        \For{$k=0,\dots$}\Comment{Phase III}
        \State compute $y^k_h,p^k_h,u^k_h,\mathrm{conv}$ using \cref{alg:ssn} for $\rho_k$, $y^{k-1}_h,p^{k-1}_h,u^{k-1}_h$
        \State compute $d_k = \norm{y_k - y^\delta}_Y$
        \If{$\delta \leq d_k\leq \tau\delta$ and $\mathrm{conv}=\mathrm{true}$}
        \State \textbf{break}
        \ElsIf{$d_k> \tau\delta$ or $\mathrm{conv}=\mathrm{false}$}
        \State set $\Delta\rho \gets \Delta\rho/2$
        \State set $\rho_{k+1} = \rho_k + \Delta\rho$
        \Else 
        \State set $\rho_{k+1} = \rho_k - \Delta\rho$
        \EndIf
        \EndFor
        \Ensure{$y^k_h,p^k_h,u^k_h$, $\rho_k$}
    \end{algorithmic}
\end{algorithm}

\section{Numerical example}\label{sec:example}

We illustrate the regularization properties of the quasi-solution using a numerical example for the setting considered in \cref{sec:solution}, which is a moderately ill-posed problem. 
Specifically, we take $\Omega=[-1,1]^2\subset \R^2$ and discretize it using a uniform (Friedrichs--Keller) triangulation with $N=128\times 128$ vertices. The potential coefficient is fixed at $c=1$. We then choose a piecewise constant true parameter $u^\dag$ consisting of inclusions of different (positive and negative) heights with $\rho^\dag:=\norm{u^\dag}_{L^{\infty}(\Omega)}=4$; see \cref{fig:recon:true}. (For the sake of conciseness, we omit the subscript $h$ in the following.) From this, noisy data is generated by setting
\begin{equation*}
    y^\delta = y^\dag + \frac{s}{100}\norm{y^\dag}_\infty \frac{\eta}{\norm{\eta}_{2}},
\end{equation*}
where $y^\dag = A_h u^\dag$, $s>0$ is a chosen relative noise percentage, $\eta\in \R^N$ is a random vector with $\eta_i$, $1\leq i\leq N$, independently normally distributed with mean $0$ and standard deviation $1$, and $\norm{\eta}_2 := \eta^TM_h \eta$ is the discrete $L^2(\Omega)$-norm.

We then compute the quasi-solution $u^\delta_\rho$ using a Matlab implementation of the procedure discussed in \cref{sec:solution}, where the parameters are chosen as follows. In \cref{alg:ssn}, we set $q=0.7$, $i_{\max}=10$, $k_{\max}=30$ and $\mathrm{tol}=10^{-9}$. In \cref{alg:param}, we set $\tau = 1.1$ and $\rho_0=10$.

\Crefrange{fig:recon:1}{fig:recon:0001} show the quasi-solutions for $s = 1,0.02,0.01,0.001,0.0001$ with corresponding $\rho=\rho(\delta,y^\delta)\approx 2.5,3.896,3.952,3.995,3.999$ according to the discrepancy principle. First, we note that in all cases $\norm{u^\delta_\rho}_{L^\infty(\Omega)} = \rho$ as expected from \cref{cor:quasisol_boundary}. For $s=1$, we have in fact that $u^\delta_\rho(x) = \rho$ almost everywhere, which leads to a poor reconstruction quality. However, as $s$ and therefore $\delta:=\norm{y^\delta-y^\dag}_{2}$ decreases, the quality increases until the quasi-solution is virtually identical to $u^\dag$ for $s=0.0001$. While the reconstruction only becomes really acceptable for $s<0.01$, two properties are of note: First, ignoring the noisy background, the location and shape of the inclusions is recovered well even for large $s$; in particular, no smoothing is visible in any of the reconstructions. Second, this is especially the case for the inclusions of largest magnitude for which $u^\dag_\rho(x) = \rho^\dag$, for which location and shape are reconstructed perfectly (as can be seen -- with a bit of effort -- even for $s=1$ in \cref{fig:recon:1}). This ties in with \cref{rem:bregman}, which indicates improved convergence behavior at these points. 
\begin{figure}[htp]
    \centering
    \begin{subfigure}[t]{0.495\textwidth}
        \centering
        \includegraphics[width=\textwidth]{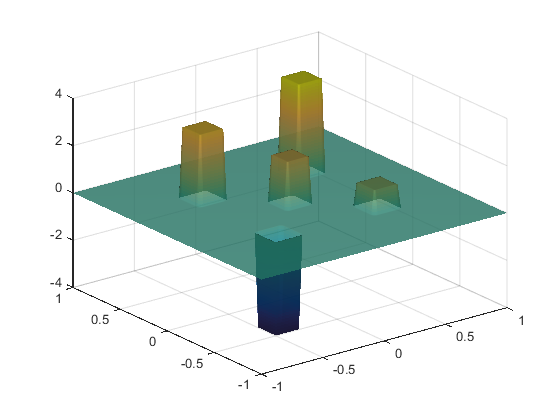}
        \caption{$u^\dagger$ with $\norm{u^\dag}_{L^\infty(\Omega)} = 4$}\label{fig:recon:true}
    \end{subfigure}
    \begin{subfigure}[t]{0.495\textwidth}
        \centering
        \includegraphics[width=\textwidth]{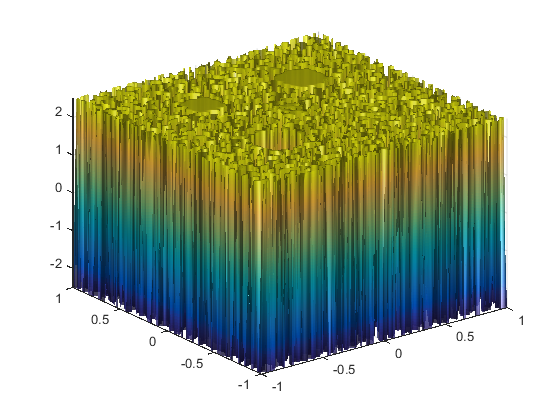}
        \caption{$u^\delta_\rho$ for $s=1\%$, $\rho \approx 2.5$}
        \label{fig:recon:1}
    \end{subfigure}
    \begin{subfigure}[t]{0.495\textwidth}
        \centering
        \includegraphics[width=\textwidth]{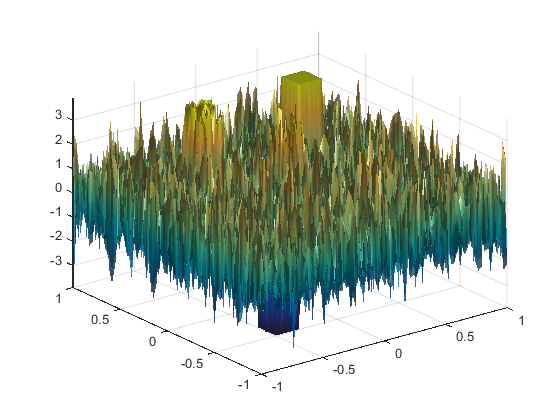}
        \caption{$u^\delta_\rho$ for $s=0.02$, $\rho \approx 3.896$}
        \label{fig:recon:02}
    \end{subfigure}
    \begin{subfigure}[t]{0.495\textwidth}
        \centering
        \includegraphics[width=\textwidth]{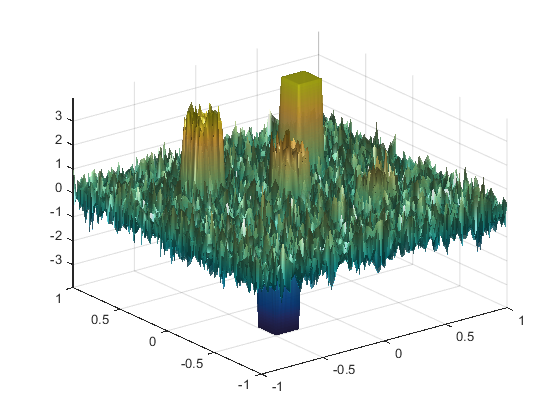}
        \caption{$u^\delta_\rho$ for $s = 0.01$, $\rho \approx 3.952$}
        \label{fig:recon:01}
    \end{subfigure}
    \begin{subfigure}[t]{0.495\textwidth}
        \centering
        \includegraphics[width=\textwidth]{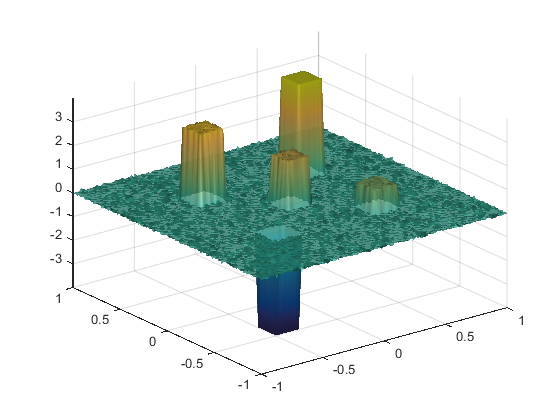}
        \caption{$u^\delta_\rho$ for $s=0.001$, $\rho \approx 3.995$}
        \label{fig:recon:001}
    \end{subfigure}
    \begin{subfigure}[t]{0.495\textwidth}
        \centering
        \includegraphics[width=\textwidth]{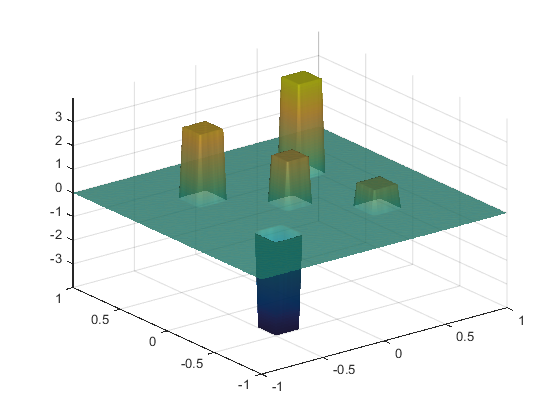}
        \caption{$u^\delta_\rho$ for $s = 0.0001$, $\rho \approx 3.999$}
        \label{fig:recon:0001}
    \end{subfigure}
    \caption{true parameter $u^\dagger$ and quasi-solutions $u^\delta_\rho$ for different noise percentages}\label{fig:recon}
\end{figure}

This is further illustrated by quantitative results for a larger selection of values of the noise percentage $s$ in \cref{tab:errors}, where we give the actual noise level $\delta$, the discrepancy $d(\rho,y^\delta)$ achieved by the parameter choice \cref{alg:param}, and the corresponding reconstruction error of the quasi-solution. For the sake of completeness, we give here both the $L^2(\Omega)$-error $\norm{u^\delta_\rho-u^\dag}_2$ and the $L^{\infty}(\Omega)$-error $\norm{u^\delta_\rho-u^\dag}_\infty$ as well as the error of the duality pairing $\dual{\xi,u^\delta_\rho-u^\dag}$ in the discrete Bregman distance $D^\xi_J(u^\delta_\rho;u^\dag)$ for $\xi\in\R^N$ constructed according to \cref{rem:bregman}, i.e.,
\begin{equation}
    \xi_i = 
    \begin{cases} 
        \phantom{-}d & \text{if }u^\dag_i = \rho^\dag,\\
        \phantom{-}0 & \text{if }|u^\dag_i| < \rho^\dag,\\
        {-}d & \text{if }u^\dag_i = -\rho^\dag,
    \end{cases}
    \qquad\text{for}\qquad d=\left|\setof{i}{|u^\dag_i| = \rho^\dag}\right|.
\end{equation}
(Note that in the discrete setting, $A_h$ is invertible and hence $\xi\in R(A_h^*)$.) It can be observed that at least for $s\leq 0.01$, all these errors behave as $\mathcal{O}(\delta)$, corroborating both \cref{thm:conv_rates} and the observed pointwise convergence.
\begin{table}[b]
    \centering
    \caption{quantitative results: noise percentage $s$, effective noise level $\delta$, reconstruction errors in $L^2(\Omega)$- and $L^\infty(\Omega)$-norms and related to Bregman distance $D^\xi_J$}
    \label{tab:errors}
    \begin{tabular}{%
            S[table-format=1e-1]
            S[table-format=1.3e-1]
            S[table-format=1.3e-1]
            S[table-format=1.3]
            S[table-format=1.3e-1]
            S[table-format=1.3e-1]
            S[table-format=1.3e-1]
        }
        \toprule
        {s} & {$\delta$} & {$d(\rho,y^\delta)$} & {$\rho(\delta,y^\delta)$} & {$\norm{u^{\delta}_{\rho}-u^\delta}_{\infty}$} & {$\norm{u^{\delta}_{\rho}-u^\delta }_{2}$} & {$\dual{\xi,u^{\delta}_{\rho}-u^\dagger}$} \\
        \midrule
        e-00 & 1.310e-03 & 1.321e-03 & 2.500 & 4.500e+00 & 3.571e+00 & 1.500e+00        \\
        e-01 & 1.310e-04 & 1.400e-04 & 3.594 & 5.594e+00 & 3.101e+00 & 4.063e-01        \\
        e-02 & 1.310e-05 & 1.358e-05 & 3.952 & 2.894e+00 & 5.620e-01 & 4.836e-02        \\
        e-03 & 1.310e-06 & 1.367e-06 & 3.995 & 2.051e-01 & 5.661e-02 & 4.944e-03        \\
        e-04 & 1.310e-07 & 1.391e-07 & 3.999 & 2.317e-02 & 5.726e-03 & 4.987e-04        \\
        e-05 & 1.310e-08 & 1.357e-08 & 4.000 & 2.070e-03 & 5.612e-04 & 4.844e-05        \\
        \bottomrule
    \end{tabular}
\end{table}

\section{Conclusion}

We have extended the method of quasi-solutions, or Ivanov regularization, to non-reflexive Banach spaces and shown weak* stability and convergence as well as convergence rates in Bregman distances. In particular, we have characterized parameter choice rules that lead to a convergent regularization method. While it turns out that in this setting, a true a priori choice is not feasible, the classical a posteriori choice according to the discrepancy principle is possible and can be exploited as a continuation strategy for the efficient numerical computation of quasi-solutions to inverse source problems using a Newton-type method. Numerical examples illustrate the regularization properties of this approach.

The results in \cref{sec:example} show that the method of quasi-solutions is indeed a very weak regularization. On the one hand, this is beneficial since it doesn't introduce additional smoothing as Tikhonov regularization with an $L^2(\Omega)$ penalty would. On the other hand, the reconstructions show large residual noise unless the data noise is relatively small. This indicates that at least for inverse source problems, Ivanov regularization is likely too weak for practical application in general; an exception would be problems where it is expected that $u^\dag\in\{-\rho^\dag,\rho^\dag\}$ almost everywhere or where the main interest lies in the correct localization of strong inclusions.

Further study of quasi-solutions in Banach spaces would therefore be justified. Besides the extension to nonlinear parameter identification problems for PDEs mentioned in the introduction, it would be interesting to consider more general choices of $M_\rho$ such as, e.g., $M_\rho=\setof{u}{0\leq u(x)\leq \rho}$ or $M_\rho=\setof{u}{\rho^{-1}\leq u(x) \leq \rho}$; the former would be relevant for the problem of identifying a potential coefficient in an elliptic PDE (and furnish a link to sparsity regularization), while the latter would be appropriate for identification of a diffusion coefficient. Finally, it is an open question whether heuristic parameter choice rules for $\rho$ are possible, e.g., based on a model function approach for the distance function $d(\rho,y^\delta)$.

\section*{Acknowledgments}

This work was supported by the German Science Foundation (DFG) under grant CL 487/1-1. 
The authors wish to thank Barbara Kaltenbacher for helpful remarks and discussions.

\printbibliography

\end{document}